\documentclass[oneside,english]{amsart}
\usepackage[T1]{fontenc}
\usepackage[latin9]{inputenc}
\usepackage{geometry}
\usepackage{amstext}
\usepackage{amsthm}
\usepackage{amssymb}

\makeatletter
\numberwithin{equation}{section}
\numberwithin{figure}{section}
\theoremstyle{plain}
\newtheorem{thm}{\protect\theoremname}
  \theoremstyle{plain}
  \newtheorem{lem}[thm]{\protect\lemmaname}

\usepackage{babel}

\usepackage{babel}

  \providecommand{\lemmaname}{Lemma}
\providecommand{\theoremname}{Theorem}

\usepackage{babel}

  \providecommand{\lemmaname}{Lemma}
\providecommand{\theoremname}{Theorem}

\usepackage{babel}

  \providecommand{\lemmaname}{Lemma}
\providecommand{\theoremname}{Theorem}

\usepackage{babel}
\providecommand{\lemmaname}{Lemma}
\providecommand{\theoremname}{Theorem}

\makeatother

\usepackage{babel}
  \providecommand{\lemmaname}{Lemma}
\providecommand{\theoremname}{Theorem}

\begin{document}

\title{Zeros of polynomials with four-term recurrence}

\author{Khang Tran \and Andres Zumba}
\begin{abstract}
For any real numbers $b,c\in\mathbb{R}$, we form the sequence of
polynomials $\left\{ H_{m}(z)\right\} _{m=0}^{\infty}$ satisfying
the four-term recurrence 
\[
H_{m}(z)+cH_{m-1}(z)+bH_{m-2}(z)+zH_{m-3}(z)=0,\qquad m\ge3,
\]
with the initial conditions $H_{0}(z)=1$, $H_{1}(z)=$$-c$, and
$H_{2}(z)=-b+c^{2}$. We find necessary and sufficient conditions
on $b$ and $c$ under which the zeros of $H_{m}(z)$ are real for
all $m$, and provide an explicit real interval on which ${\displaystyle \bigcup_{m=0}^{\infty}\mathcal{Z}(H_{m})}$
is dense where $\mathcal{Z}(H_{m})$ is the set of zeros of $H_{m}(z)$. 
\end{abstract}

\maketitle

\section{Introduction}

Consider the sequence of polynomials $\left\{ H_{m}(z)\right\} _{m=0}^{\infty}$
satisfying a finite recurrence 
\begin{equation}
\sum_{k=0}^{n}a_{k}(z)H_{m-k}(z)=0,\qquad m\ge n,\label{eq:generalrecurrence}
\end{equation}
where $a_{k}(z)$, $1\le k\le n$, are complex polynomials. With certain
initial conditions, one may ask for the locations of the zeros of
$H_{m}(z)$ on the complex plane. There are two common approaches:
asymptotic location and exact location. The first direction asks for
a curve where the zeros of $H_{m}(z)$ approach as $m\rightarrow\infty$,
and the second direction aims at a curve where the zeros of $H_{m}(z)$
lie exactly on for all $m$ (or at least for all large $m$). Recent
works in the first direction include \cite{bkw,bkw-1,bbs,bg,bg-1}.
Results in this first direction are useful to prove the necessary
condition for the reality of zeros of $H_{m}(z)$ that we will see
in Section 3. 

It is not an easy problem to find an explicit curve (if such exists)
where the zeros of all $H_{m}(z)$ lie for a general recurrence in
\eqref{eq:generalrecurrence}. For three-term recurrences with degree
two and appropriate initial conditions, the curve containing zeros
is given in \cite{tran}. The corresponding curve for a three-term
recurrence with degree $n$ is given in \cite{tran-1}. Among all
possible curves containing the zeros of the $H_{m}(z)$s, the real
line plays an important role. We say that a polynomial is hyperbolic
if all of its zeros are real. There are a lot of recent works on hyperbolic
polynomials and on linear operators preserving hyperbolicity of polynomials,
see for example \cite{bb,fp}. In some cases, we could find the curve
on the plane containing the zeros of a sequence of polynomials by
claiming certain polynomials are hyperbolic, see for example \cite{ft-1}.

The main result of this paper is the identification of necessary and
sufficient conditions on $b,c\in\mathbb{R}$ under which the zeros
of the sequence of polynomials $H_{m}(z)$ satisfying the recurrence
\begin{equation}
\begin{aligned}H_{m}(z)+cH_{m-1}(z)+bH_{m-2}(z)+zH_{m-3}(z) & =0,\qquad m\ge3,\\
H_{0}(z) & \equiv1,\\
H_{1}(z) & \equiv-c,\\
H_{2}(z) & \equiv-b+c^{2},
\end{aligned}
\label{eq:recurrence}
\end{equation}
are real. We use the convention that the zeros of the constant zero
polynomial are real. Let $\mathcal{Z}(H_{m})$ denote the set of zeros
of $H_{m}(z)$. 
\begin{thm}
\label{thm:mainstatement}Suppose $b,c\in\mathbb{R}$, and let $\left\{ H_{m}(z)\right\} _{m=0}^{\infty}$
be defined as in \eqref{eq:recurrence}. The zeros of $H_{m}(z)$
are real for all $m$ if and only if one of the two conditions below
holds

(i) $c=0$ and $b\ge0$

(ii) $c\ne0$ and $-1\le b/c^{2}\le1/3$. \\
 Moreover, in the first case with $b>0$, ${\displaystyle \bigcup_{m=0}^{\infty}\mathcal{Z}(H_{m})}$
is dense on $(-\infty,\infty)$. In the second case, ${\displaystyle \bigcup_{m=0}^{\infty}\mathcal{Z}(H_{m})}$
is dense on the interval 
\[
c^{3}.\left(-\infty,\frac{-2+9b/c^{2}-2\sqrt{(1-3b/c^{2})^{3}}}{27}\right].
\]
\end{thm}
Our paper is organized as follow. In Section 2, we prove the sufficient
condition for the reality of the zeros of all $H_{m}(z)$ in the case
$c\ne0$. The case $c=0$ follows from similar arguments whose key
differences will be mentioned in Section 3. Finally, in Section 4,
we prove the necessary condition for the reality of the zeros of $H_{m}(z)$.

\section{The case $c\protect\ne0$ and $-1\le b/c^{2}\le1/3$}

We write the sequence $H_{m}(z)$ in \eqref{eq:recurrence} using
its generating function 
\begin{equation}
\sum_{m=0}^{\infty}H_{m}(z)t^{m}=\frac{1}{1+ct+bt^{2}+zt^{3}}.\label{eq:genfuncgeneral}
\end{equation}
From \eqref{eq:genfuncgeneral}, if we make the substitutions $t\rightarrow t/c$,
$b/c^{2}\rightarrow a$, and $z/c^{3}\rightarrow z$, it suffices
to prove the following form of the theorem. 
\begin{thm}
\label{thm:cneq0}Consider the sequence of polynomials $\left\{ H_{m}(z)\right\} _{m=0}^{\infty}$
generated by 
\begin{equation}
\sum_{m=0}^{\infty}H_{m}(z)t^{m}=\frac{1}{1+t+at^{2}+zt^{3}}\label{eq:genfunc}
\end{equation}
where $a\in\mathbb{R}.$ If $-1\le a\le1/3$ then the zeros of $H_{m}(z)$
lie on the real interval 
\begin{equation}
I_{a}=\left(-\infty,\frac{-2+9a-2\sqrt{(1-3a)^{3}}}{27}\right]\label{eq:intervaldef}
\end{equation}
and ${\displaystyle \bigcup_{m=0}^{\infty}\mathcal{Z}(H_{m})}$ is
dense on $I_{a}$.
\end{thm}
We will see later that the density of the union of zeros on $I_{a}$
follows naturally from the proof that $\mathcal{Z}(H_{m})\subset I_{a}$
and thus we focus on proving this claim. We note that each value of
$a\in[-1,1/3]$ generates a sequence of polynomials $H_{m}(z,a)$.
The lemma below asserts that it suffices to prove that $\mathcal{Z}(H_{m}(z,a))\subset I_{a}$
for all $a$ in a dense subset of $[-1,1/3]$. 
\begin{lem}
\label{lem:densesubint}Let $S$ be a dense subset of $[-1,1/3]$,
and let $m\in\mathbb{N}$ be fixed. If
\[
\mathcal{Z}(H_{m}(z,a))\in I_{a}
\]
for all $a\in S$ then 
\[
\mathcal{Z}(H_{m}(z,a^{*}))\in I_{a^{*}}
\]
for all $a^{*}\in[-1,1/3]$.
\end{lem}
\begin{proof}
Let $a^{*}\in[-1,1/3]$ be given. By the density of $S$ in $[-1,1/3]$,
we can find a sequence $\{a_{n}\}$ in $S$ such that $a_{n}\rightarrow a^{*}$.
For any $z^{*}\notin I_{a^{*}}$, we will show that $H_{m}(z^{*},a^{*})\ne0$.
We note that the zeros of $H_{m}(z,a_{n})$ lie on the interval $I_{a_{n}}$
whose right endpoint approaches the right endpoint of $I_{a^{*}}$
as $n\rightarrow\infty$. If we let $z_{k}^{(n)}$, $1\le k\le\deg H_{m}(z,a_{n})$,
be the zeros of $H_{m}(z,a_{n})$ then 
\[
\left|H_{m}(z^{*},a_{n})\right|=\gamma^{(n)}\prod_{k=1}^{\deg H_{m}(z,a_{n})}\left|z^{*}-z_{k}^{(n)}\right|
\]
where $\gamma^{(n)}$ is the lead coefficient of $H_{m}(z,a_{n})$.
We will see in the next lemma that $\deg H_{m}(z,a_{n})$ is at most
$\left\lfloor m/3\right\rfloor $. From this finite product and the
assumption that $z^{*}\notin I_{a}$, we conclude that there is a
fixed (independent of $n$) $\delta>0$ so that $|H_{m}(z^{*},a_{n})|>\delta$,
for all large $n$. Since $H_{m}(z^{*},a)$ is a polynomial in $a$
for any fixed $z^{*}$, we conclude that 
\[
H_{m}(z^{*},a^{*})=\lim_{n\rightarrow\infty}H_{m}(z^{*},a_{n})\ne0
\]
and the lemma follows.
\end{proof}
Lemma \ref{lem:densesubint} allows us to ignore some special values
of $a$. In particular, we may assume $a\ne0$. In our main approach,
we count the number of zeros of $H_{m}(z)$ on the interval $I_{a}$
in \eqref{eq:intervaldef} and show that this number of zeros is at
least the degree of $H_{m}(z)$. To count the number of zeros of $H_{m}(z)$
on $I_{a}$, we write $z=z(\theta)$ as a strictly increasing function
of a variable $\theta$ on the interval $(2\pi/3,\pi)$. Then we construct
a function $g_{m}(\theta)$ on $(2\pi/3,\pi)$ with the property that
$\theta$ is a zero of $g_{m}(\theta)$ on $(2\pi/3,\pi)$ if and
only if $z(\theta)$ is a zero of $H_{m}(z)$ on $I_{a}$. From this
construction, we count the number of zeros of $g_{m}(\theta)$ on
$(2\pi/3,\pi)$ which will be the same as the number of zeros of $H_{m}(z)$
on $I_{a}$ by the monotonicity of the function $z(\theta)$. We first
obtain an upper bound for the degree of $H_{m}(z)$ and provide heuristic
arguments for the formulas of $z(\theta)$ and $g_{m}(\theta)$. 
\begin{lem}
\label{degreelem}The degree of the polynomial $H_{m}(z)$ defined
by \eqref{eq:genfunc} is at most $\left\lfloor m/3\right\rfloor $. 
\end{lem}
\begin{proof}
We rewrite \eqref{eq:genfunc} as 
\[
(1+t+at^{2}+zt^{3})\sum_{m=0}^{\infty}H_{m}(z)t^{m}=1.
\]
By equating the coefficients in $t$ of both sides, we see that the
sequence $\left\{ H_{m}(z)\right\} _{m=0}^{\infty}$ satisfies the
recurrence 
\[
H_{m+3}(z)+H_{m+2}(z)+aH_{m+1}(z)+zH_{m}(z)=0
\]
and the initial condition $H_{0}(z)\equiv1$, $H_{1}(z)\equiv-1$,
and $H_{2}(z)\equiv1-a$. The lemma follows from induction.
\end{proof}

\subsection{Heuristic arguments}

We now provide heuristic arguments to motivate the formulas for two
special functions $z(\theta)$ and $g_{m}(\theta)$ on $(2\pi/3,\pi)$.
Let $t_{0}=t_{0}(z)$, $t_{1}=t_{1}(z)$, and $t_{2}=t_{2}(z)$ be
the three zeros of the denominator $1+t+at^{2}+zt^{3}$. We will show
rigorously in Section 2.2 that $t_{0}$, $t_{1}$, $t_{2}$ are nonzero
and distinct with $t_{0}=\overline{t_{1}}$. We let $q=t_{1}/t_{0}=e^{2i\theta}$,
$\theta\ne0,\pi$. We have 
\[
\sum_{m=0}^{\infty}H_{m}(z)t^{m}=\frac{1}{1+t+at^{2}+zt^{3}}=\frac{1}{z(t-t_{0})(t-t_{1})(t-t_{2})}.
\]
We apply partial fractions to rewrite the generating functions $1/z(t-t_{0})(t-t_{1})(t-t_{2})$
as 
\[
\frac{1}{z(t-t_{0})(t_{0}-t_{1})(t_{0}-t_{2})}+\dfrac{1}{z(t-t_{1})(t_{1}-t_{0})(t_{1}-t_{2})}+\dfrac{1}{z(t-t_{2})(t_{2}-t_{0})(t_{2}-t_{1})}
\]
which can be expanded as a series in $t$ below 
\begin{equation}
-\sum_{m=0}^{\infty}\frac{1}{z}\left(\frac{1}{(t_{0}-t_{1})(t_{0}-t_{2})t_{0}^{m+1}}+\frac{1}{(t_{1}-t_{0})(t_{1}-t_{2})t_{1}^{m+1}}+\frac{1}{(t_{2}-t_{0})(t_{2}-t_{1})t_{2}^{m+1}}\right)t^{m}.\label{eq:partialfractions}
\end{equation}
From this expression, we deduce that $z$ is a zero of $H_{m}(z)$
if and only if 
\begin{equation}
\frac{1}{(t_{0}-t_{1})(t_{0}-t_{2})t_{0}^{m+1}}+\frac{1}{(t_{1}-t_{0})(t_{1}-t_{2})t_{1}^{m+1}}+\frac{1}{(t_{2}-t_{0})(t_{2}-t_{1})t_{2}^{m+1}}=0.\label{eq:partialfraczero}
\end{equation}
After multiplying the left side of \eqref{eq:partialfraczero} by
$t_{0}^{m+3}$ we obtain the equality 
\[
\frac{1}{(1-t_{1}/t_{0})(1-t_{2}/t_{0})}+\frac{1}{(t_{1}/t_{0}-1)(t_{1}/t_{0}-t_{2}/t_{0})(t_{1}/t_{0})^{m+1}}+\frac{1}{(t_{2}/t_{0}-1)(t_{2}/t_{0}-t_{1}/t_{0})(t_{2}/t_{0})^{m+1}}=0.
\]
Setting $\zeta=t_{2}/t_{0}e^{i\theta}$, we rewrite the left side
as
\[
\frac{1}{(1-e^{2i\theta})(1-\zeta e^{i\theta})}+\frac{1}{(e^{2i\theta}-1)(e^{2i\theta}-\zeta e^{i\theta})(e^{2i\theta})^{m+1}}+\frac{1}{(\zeta e^{i\theta}-1)(\zeta e^{i\theta}-e^{2i\theta})(\zeta e^{i\theta})^{m+1}},
\]
or equivalently 
\[
\frac{1}{e^{2i\theta}(-2i\sin\theta)(e^{-i\theta}-\zeta)}+\frac{1}{(2i\sin\theta)(e^{i\theta}-\zeta)(e^{2i\theta})^{m+2}}+\frac{1}{(\zeta-e^{-i\theta})(\zeta-e^{i\theta})(\zeta)^{m+1}(e^{i\theta})^{m+3}}.
\]
We multiply this expression by $(\zeta-e^{-i\theta})(\zeta-e^{i\theta})e^{i(m+3)\theta}$
and set the summation to $0$ and rewrite \eqref{eq:partialfraczero}
as 
\begin{align}
0 & =\frac{(\zeta-e^{i\theta})e^{i(m+1)\theta}}{2i\sin\theta}+\frac{e^{-i\theta}-\zeta}{(2i\sin\theta)e^{i(m+1)\theta}}+\frac{1}{\zeta^{m+1}}\nonumber \\
 & =\frac{(\zeta-e^{i\theta})e^{i(m+1)\theta}-(\zeta-e^{-i\theta})e^{-i(m+1)\theta}}{2i\sin\theta}+\frac{1}{\zeta^{m+1}}\nonumber \\
 & =\frac{\zeta(e^{i(m+1)\theta}-e^{-i(m+1)\theta})+e^{-i(m+2)\theta}-e^{i(m+2)\theta}}{2i\sin\theta}+\frac{1}{\zeta^{m+1}}\nonumber \\
 & =\frac{\zeta(2i\sin(m+1)\theta)-2i\sin(m+2)\theta}{2i\sin\theta}+\frac{1}{\zeta^{m+1}}\nonumber \\
 & =\frac{2i\zeta\sin(m+1)\theta-2i\sin(m+1)\theta\cos\theta-2i\cos(m+1)\theta\sin\theta}{2i\sin\theta}+\frac{1}{\zeta^{m+1}}\nonumber \\
 & =\frac{(\zeta-\cos\theta)\sin(m+1)\theta}{\sin\theta}-\cos(m+1)\theta+\frac{1}{\zeta^{m+1}}.\label{eq:gthetamotivation}
\end{align}
The last expression will serve as the definition of $g_{m}(\theta)$. 

We next provide a motivation for the specific form of $z(\theta)$.
Since $t_{0}$, $t_{1}$, and $t_{2}$ are the zeros of $D(t,z)=1+t+at^{2}+zt^{3}$,
they satisfy the three identities 
\begin{align*}
t_{0}+t_{1}+t_{2} & =-\frac{a}{z},\\
t_{0}t_{1}+t_{0}t_{2}+t_{1}t_{2} & =\frac{1}{z},\qquad\text{and}\\
t_{0}t_{1}t_{2} & =-\frac{1}{z}.
\end{align*}
If we divide the first equation by $t_{0}$, the second by $t_{0}^{2}$,
and the third by $t_{0}^{3}$ then these identities become 
\begin{align}
1+e^{2i\theta}+\zeta e^{i\theta} & =-\frac{a}{zt_{0}},\label{eq:elemsym1}\\
e^{2i\theta}+\zeta e^{i\theta}+\zeta e^{3i\theta} & =\frac{1}{zt_{0}^{2}},\qquad\text{and}\label{eq:elemsym2}\\
\zeta e^{3i\theta} & =-\frac{1}{zt_{0}^{3}}.\label{eq:elemsym3}
\end{align}
We next divide the first identity by second, and the second by the
third to obtain 
\begin{align*}
\frac{1+e^{2i\theta}+\zeta e^{i\theta}}{e^{2i\theta}+\zeta e^{i\theta}+\zeta e^{3i\theta}} & =-at_{0},\qquad\text{and}\\
\frac{e^{2i\theta}+\zeta e^{i\theta}+\zeta e^{3i\theta}}{\zeta e^{3i\theta}} & =-t_{0},
\end{align*}
from which we deduce that 
\[
(1+e^{2i\theta}+\zeta e^{i\theta})\zeta e^{3i\theta}=a(e^{2i\theta}+\zeta e^{i\theta}+\zeta e^{3i\theta})^{2}.
\]
This equation is equivalent to 
\[
(e^{-i\theta}+e^{i\theta}+\zeta)\zeta e^{4i\theta}=ae^{4i\theta}(1+\zeta e^{-i\theta}+\zeta e^{i\theta})^{2},
\]
or simply 
\[
(2\cos\theta+\zeta)\zeta=a(1+2\zeta\cos\theta)^{2}.
\]

\begin{lem}
\label{lem:realzeta}For any $a\in[-1,1/3]$ and $\theta\in(2\pi/3,\pi)$,
the zeros in $\zeta$ of the polynomial 
\begin{equation}
(2\cos\theta+\zeta)\zeta-a(1+2\zeta\cos\theta)^{2}\label{eq:quadraticzetatheta}
\end{equation}
are real and distinct. 
\end{lem}
\begin{proof}
We consider the discriminant of the above polynomial in $\zeta$:
\[
\Delta=(1-4a)\cos^{2}\theta+a.
\]
There are three possible cases depending on the value of $a$. If
$1/4\leq a\leq1/3$, the inequality $\Delta>0$ comes directly from
$a\geq4a-1>(1-4a)\cos^{2}\theta$. If $0\leq a<1/4$, the claim $\Delta>0$
is trivial since $1-4a>0$. Finally, if $a<0$, we have $\Delta\geq(1-4a)/4+a=1/4$.
It follows that the zeros of \eqref{eq:quadraticzetatheta} are real
and distinct for any $a\in[-1,1/3]$ and $\theta\in(2\pi/3,\pi)$. 
\end{proof}
To obtain the formula $z(\theta)$, we multiply both sides of \eqref{eq:elemsym1}
and \eqref{eq:elemsym2} 
\[
(1+e^{2i\theta}+\zeta e^{i\theta})(e^{2i\theta}+\zeta e^{i\theta}+\zeta e^{3i\theta})=-\frac{a}{z^{2}t_{0}^{3}}
\]
and divide by \eqref{eq:elemsym3} to arrive at the resulting equation
\begin{align}
z & =\frac{ae^{3i\theta}\zeta}{(1+e^{2i\theta}+\zeta e^{i\theta})(e^{2i\theta}+\zeta e^{i\theta}+\zeta e^{3i\theta})}\nonumber \\
 & =\frac{ae^{3i\theta}\zeta}{e^{3i\theta}(e^{-i\theta}+e^{i\theta}+\zeta)(1+\zeta e^{-i\theta}+\zeta e^{i\theta})}\nonumber \\
 & =\dfrac{a\zeta}{(2\cos\theta+\zeta)(1+2\zeta\cos\theta)}.\label{eq:zthetamotivation}
\end{align}

\subsection{Rigorous proof}

Motivated by Section 2.1, we will rigorously prove Theorem \eqref{thm:cneq0}
in this Section 2.2. We start our proof of Theorem \eqref{thm:cneq0}
by defining the function $\zeta(\theta)$ according to \eqref{eq:quadraticzetatheta}:
\begin{equation}
\zeta=\zeta(\theta)=\frac{(2a-1)\cos\theta+\sqrt{(1-4a)\cos^{2}\theta+a}}{1-4a\cos^{2}\theta}\label{eq:zetaform}
\end{equation}
which, from Lemma \ref{lem:realzeta}, is a real function on $(2\pi/3,\pi)$
with a possible vertical asymptote at 
\begin{equation}
\theta=\cos^{-1}\left(-\frac{1}{2\sqrt{a}}\right)\label{eq:vertasymp}
\end{equation}
when $1/4<a\le1/3$. However we note that the function $1/\zeta(\theta)$
is a real continuous function on $(2\pi/3,\pi)$. 
\begin{lem}
\label{lem:outsidedisk} Let $\zeta(\theta)$ be defined as in \eqref{eq:zetaform}.
Then $|\zeta(\theta)|>1$ for every $a\in(-1,1/3)$ and every $\theta\in(2\pi/3,\pi)$
with $1-4a\cos^{2}\theta\ne0$.
\end{lem}
\begin{proof}
From \eqref{eq:quadraticzetatheta}, we note that $\zeta_{+}:=\zeta(\theta)$
and 
\[
\zeta_{-}:=\frac{(2a-1)\cos\theta-\sqrt{(1-4a)\cos^{2}\theta+a}}{1-4a\cos^{2}\theta}
\]
are the zeros of 
\[
f(\zeta):=(2\cos\theta+\zeta)\zeta-a(1+2\zeta\cos\theta)^{2}.
\]
Note that 
\[
f(-1)f(1)=(-1+2\cos\theta)(1+2\cos\theta)(4a^{2}\cos^{2}\theta-(a-1)^{2}).
\]
If $\theta\in(2\pi/3,\pi)$ and $a\in(-1,1/3)$, this product is negative
since
\[
4a^{2}\cos^{2}\theta-(a-1)^{2}\leq4a^{2}-(a-1)^{2}=(a+1)(3a-1)<0.
\]
Thus exactly one of the zeros of the quadratic function $f(\zeta)$
lies outside the interval $[-1,1]$. The claim follows from the fact
that $|\zeta_{+}|>|\zeta_{-}|$. 
\end{proof}
Although one can prove Lemma \ref{lem:outsidedisk} for the extreme
value $a=-1$ or $a=1/3$, that will not be necessary by Lemma \ref{lem:densesubint}.
Next we define the real function $z(\theta)$ by 
\begin{equation}
z=z(\theta):=\frac{a\zeta}{(2\cos\theta+\zeta)(1+2\zeta\cos\theta)}\label{eq:ztheta}
\end{equation}
on $(2\pi/3,\pi)$. From Lemma \ref{lem:outsidedisk}, $1+2\zeta\cos\theta\ne0$
and so does $2\cos\theta+\zeta$ by \eqref{eq:quadraticzetatheta}.
Dividing the numerator and the denominator by $\zeta^{2}(\theta)$
and combining with the fact that $1/\zeta(\theta)$ is continuous
on $(2\pi/3,\pi)$, we conclude that the possible discontinuity of
$z(\theta)$ in \eqref{eq:vertasymp} is removable. Finally, motivated
by \eqref{eq:gthetamotivation}, we define the function $g_{m}(\theta)$
by 
\begin{equation}
g_{m}(\theta):=\frac{(\zeta-\cos\theta)\sin(m+1)\theta}{\sin\theta}-\cos(m+1)\theta+\frac{1}{\zeta^{m+1}}\label{eq:gthetaform}
\end{equation}
which has the same vertical asymptote as that of $\zeta(\theta)$
in \eqref{eq:vertasymp} when $1/4<a\le1/3$. 

From Lemma \ref{lem:outsidedisk}, we see that the sign of the function
$g_{m}(\theta)$ alternates at values of $\theta$ where $\cos(m+1)\theta=\pm1$.
Thus by the Intermediate Value Theorem the function $g_{m}(\theta)$
has at least one root on each subinterval whose endpoints are the
solution of $\cos(m+1)=\pm1$. However, in the case $1/4\le a\le1/3$,
one of the subintervals will contain a vertical asymptote given in
\eqref{eq:vertasymp}. The lemma below counts the number of zeros
of $g_{m}(\theta)$ on such a subinterval.
\begin{lem}
\label{lem:singinterval}Let $g_{m}(\theta)$ be defined as in \eqref{eq:gthetaform}.
Suppose $1/4<a\le1/3$ and $m\ge6$. Then there exists $h\in\mathbb{N}$
such that 
\[
\theta_{h-1}:=\frac{h-1}{m+1}\pi<\cos^{-1}\left(-\frac{1}{2\sqrt{a}}\right)\le\frac{h}{m+1}\pi=:\theta_{h}
\]
where $\lfloor2(m+1)/3\rfloor+1\le h-1<h\le m+1$. Furthermore, as
long as 
\begin{equation}
\cos^{-1}\left(-\frac{1}{2\sqrt{a}}\right)\neq\frac{h}{m+1}\pi,\label{eq:differencecond}
\end{equation}
the function $g(\theta)$ has at least two zeros on the interval 
\begin{equation}
\theta\in\left(\frac{h-1}{m+1}\pi,\frac{h}{m+1}\pi\right):=J_{h}\label{eq:subintdef}
\end{equation}
whenever $h$ is at most $m$, and at least one zero when $h$ is
$m+1$. 
\end{lem}
\begin{proof}
Suppose $a\in[1/4,1/3]$. Since the function $\cos^{-1}\left(-1/2\sqrt{x}\right)$
is decreasing on the interval $[1/4,1/3]$, we conclude that 
\[
\cos^{-1}\left(-\frac{1}{2\sqrt{a}}\right)\ge\frac{5\pi}{6}.
\]
Then the existence of $h$ comes directly from the inequality 
\[
\frac{\lfloor2(m+1)/3\rfloor+1}{m+1}\pi<\frac{5\pi}{6}
\]
when $m\ge6$. 

The vertical asympote at $\cos^{-1}(-1/2\sqrt{a})$ of $g_{m}(\theta)$
divides the interval $J_{h}$ in \eqref{eq:subintdef} into two subintervals.
We will show that each subinterval contains at least a zero of $g_{m}(\theta)$
if $h\le m$. In the case $h=m+1$, only the subinterval on the left
contains at least zero of $g_{m}(\theta)$. We analyze these two subintervals
in the two cases below.

We consider the first case when $\theta\in I_{h}$ and $\theta<\cos^{-1}(-1/2\sqrt{a})$.
By Lemma \ref{lem:outsidedisk} and \eqref{eq:gthetaform} we see
that the sign of $g_{m}(\theta_{h-1})$ is $(-1)^{h}$. We now show
that the sign of $g_{m}(\theta)$ is $(-1)^{h-1}$ when $\theta\rightarrow\cos^{-1}(-1/2\sqrt{a})$.
From \eqref{eq:zetaform}, we observe that $\zeta(\theta)\rightarrow+\infty$
as $\theta\rightarrow\cos^{-1}(-1/2\sqrt{a})$. Since $\theta\in I_{h}$,
the sign of $\sin(m+1)\theta$ is $(-1)^{h-1}$ and consequently the
sign of $g_{m}(\theta)$ is $(-1)^{h-1}$ when $\theta\rightarrow\cos^{-1}(-1/2\sqrt{a})$
by \eqref{eq:gthetaform}. By the Intermediate Value Theorem, we obtain
at least one zero of $g_{m}(\theta)$ in this case.

Next we consider the case when $\theta\in I_{h}$ and $\theta>\cos^{-1}(-1/2\sqrt{a})$.
In this case the sign of $g_{m}(\theta_{h})$ is $(-1)^{h-1}$ if
$h\le m$ by Lemma \ref{lem:outsidedisk}. Since $\zeta(\theta)\rightarrow-\infty$
as $\theta\rightarrow\cos^{-1}(-1/2\sqrt{a})$ and the sign of $\sin(m+1)\theta$
is $(-1)^{h-1}$, the sign of $g_{m}(\theta)$ is $(-1)^{h}$ as $\theta\rightarrow\cos^{-1}(-1/2\sqrt{a})$.
By the Intermediate Value Theorem, we obtain at least one zero of
$g_{m}(\theta)$ in this case and $h\le m$. 
\end{proof}
Note that we may assume \eqref{eq:differencecond} by Lemma\ref{lem:densesubint}.
From the fact that the sign of $g_{m}(\theta)$ in \eqref{eq:gthetaform}
alternates when $\cos(m+1)\theta=\pm1$, we can find a lower bound
for the number of zeros of $g_{m}(\theta)$ on $(2\pi/3,\pi)$ by
the Intermediate Value Theorem. We will relate the zeros of $g_{m}(\theta)$
to the zeros of $H_{m}(z)$ by \eqref{eq:gthetamotivation}. However
to ensure that the partial fractions procedure preceding equation
\eqref{eq:gthetamotivation} is rigorous, we need the lemma below.
\begin{lem}
\label{lem:zerosofP(t)}Let $\theta\in(0,\pi)$ such that $\theta\neq\cos^{-1}\left(-1/2\sqrt{a}\right)$
whenever $a>1/4$. The zeros in $t$ of $1+t+at^{2}+z(\theta)t^{3}$
are 
\[
t_{0}=-\dfrac{e^{2i\theta}+\zeta e^{i\theta}+\zeta e^{3i\theta}}{\zeta e^{3i\theta}},\qquad t_{1}=t_{0}e^{2i\theta},\qquad t_{2}/t_{0}=\zeta e^{i\theta}
\]
where $\zeta:=\zeta(\theta)$ is given in \eqref{eq:zetaform}.
\end{lem}
\begin{proof}
We first note that

\begin{align*}
P(t_{0}) & =1+t_{0}+at_{0}^{2}+zt_{0}^{3}\\
 & =-\frac{1}{\zeta e^{i\theta}}-e^{-2i\theta}+\frac{a}{\zeta^{2}e^{2i\theta}}\left(1+\zeta e^{-i\theta}+\zeta e^{i\theta}\right)^{2}-\frac{z}{\zeta^{3}e^{3i\theta}}\left(1+\zeta e^{-i\theta}+\zeta e^{i\theta}\right)^{3}.
\end{align*}
where $\zeta$ is a root of the quadratic equation $(2\cos\theta+\zeta)\zeta-a(1+2\zeta\cos\theta)^{2}=0$.
We apply the following identities 
\[
(1+\zeta e^{-i\theta}+\zeta e^{i\theta})^{2}=(1+2\zeta\cos\theta)^{2}=\frac{1}{a}(2\cos\theta+\zeta)\zeta=\frac{1}{a}(e^{-i\theta}+e^{i\theta}+\zeta)\zeta
\]
and 
\begin{equation}
z=\frac{a\zeta}{(2\cos\theta+\zeta)(1+2\zeta\cos\theta)}=\frac{\zeta^{2}}{(1+2\zeta\cos\theta)^{3}}=\frac{\zeta^{2}}{(1+\zeta e^{-i\theta}+\zeta e^{i\theta})^{3}}\label{eq:altzform}
\end{equation}
to conclude that $P(t_{0})=0$. Similarly, we have

\begin{align*}
P(t_{1}) & =1+t_{0}e^{2i\theta}+at_{0}^{2}e^{4i\theta}+zt_{0}^{3}e^{6i\theta}\\
 & =-\frac{e^{i\theta}}{\zeta}-e^{2i\theta}+\frac{ae^{2i\theta}}{\zeta^{2}}\left(1+\zeta e^{-i\theta}+\zeta e^{i\theta}\right)^{2}-\frac{ze^{3i\theta}}{\zeta^{3}}\left(1+\zeta e^{-i\theta}+\zeta e^{i\theta}\right)^{3}\\
 & =-\frac{e^{i\theta}}{\zeta}-e^{2i\theta}+\frac{ae^{2i\theta}}{\zeta^{2}}\frac{(e^{-i\theta}+e^{i\theta}+\zeta)\zeta}{a}-\frac{e^{3i\theta}}{\zeta^{3}}\zeta^{2}=0.
\end{align*}
Finally

\begin{align*}
P(t_{2}) & =P(\zeta t_{0}e^{i\theta})\\
 & =-\zeta e^{-i\theta}-\zeta e^{i\theta}+a\left(1+\zeta e^{-i\theta}+\zeta e^{i\theta}\right)^{2}-z\left(1+\zeta e^{-i\theta}+\zeta e^{i\theta}\right)^{3}\\
 & =-\zeta e^{-i\theta}-\zeta e^{i\theta}+a\frac{1}{a}\left(e^{-i\theta}+e^{i\theta}+\zeta\right)\zeta-\zeta^{2}=0.
\end{align*}
\end{proof}
As a consequence of Lemma \ref{lem:zerosofP(t)}, if $\theta\in(2\pi/3,\pi)$,
then the zeros of $1+t+at^{2}+z(\theta)t^{3}$ will be distinct and
$t_{1}=\overline{t_{0}}$ since $\zeta\in\mathbb{R}$ by Lemma \ref{lem:realzeta}.
Thus we can apply partial fractions given in the beginning of Section
2.1. From this partial fraction decomposition, we conclude that if
$\theta$ is a zero of $g_{m}(\theta)$, then $z(\theta)$ will be
a zero of $H_{m}(z)$. In fact, we claim that each distinct zero of
$g_{m}(\theta)$ on $(2\pi/3,\pi)$ produces a distinct zero of $H_{m}(z)$
on $I_{a}$. This is the content of the following two lemmas. 
\begin{lem}
\label{lem:zmonotone}Let $\zeta(\theta)$ be defined as in \eqref{eq:zetaform}.
The function $z(\theta)$ defined as in \eqref{eq:ztheta} is increasing
on $\theta\in(2\pi/3,\pi)$. 
\end{lem}
\begin{proof}
Lemma \ref{lem:zerosofP(t)} gives 
\[
-z=\frac{1+t_{0}+at_{0}^{2}}{t_{0}^{3}}=\frac{1+t_{1}+at_{1}^{2}}{t_{1}^{3}}.
\]
We differentiate the three terms and obtain 
\begin{equation}
dz=\frac{3+2t_{0}+at_{0}^{2}}{t_{0}^{4}}dt_{0}=\frac{3+2t_{1}+at_{1}^{2}}{t_{1}^{4}}dt_{1},\label{eq:diffeqzt}
\end{equation}
where 
\[
dt_{1}=d(t_{0}e^{2i\theta})=e^{2i\theta}dt_{0}+2it_{0}e^{2i\theta}d\theta.
\]
If we set
\[
f(t_{0})=\frac{3+2t_{0}+at_{0}^{2}}{t_{0}^{4}},\qquad f(t_{1})=\frac{3+2t_{1}+at_{1}^{2}}{t_{1}^{4}}
\]
then $f(t_{0})=\overline{f(t_{1})}$, and consequently $f(t_{0})f(t_{1})\ge0$.
Thus \eqref{eq:diffeqzt} implies that
\[
f(t_{0})dt_{0}=f(t_{1})(e^{2i\theta}dt_{0}+2it_{0}e^{2i\theta}d\theta).
\]
After solving this equation for $dt_{0}$ and substituting it into
\eqref{eq:diffeqzt}, we obtain 
\begin{equation}
\frac{dz}{d\theta}=\frac{2if(t_{0})f(t_{1})t_{0}e^{2i\theta}}{f(t_{0})-f(t_{1})e^{2i\theta}}.\label{eq:diffztheta}
\end{equation}
With $t_{0}=\tau e^{-i\theta}$, $\tau\in\mathbb{R}$, we have 
\begin{align*}
\frac{f(t_{0})-f(t_{1})e^{2i\theta}}{2it_{0}e^{2i\theta}} & =\frac{f(t_{0})e^{-i\theta}-f(t_{1})e^{i\theta}}{2it_{0}e^{i\theta}}\\
 & =\frac{\Im\left(f(t_{0})e^{-i\theta}\right)}{\tau}\\
 & =\frac{1}{\tau}\Im\left(\frac{3+2t_{0}+at_{0}^{2}}{t_{0}^{4}}e^{-i\theta}\right).
\end{align*}
We now substitute $3=-3t_{0}-3at_{0}^{2}-3zt_{0}^{3}$ and have 
\begin{align*}
\frac{f(t_{0})-f(t_{1})e^{2i\theta}}{2it_{0}e^{2i\theta}} & =\frac{1}{\tau}\Im\left(\frac{-t_{0}-2at_{0}^{2}-3zt_{0}^{3}}{t_{0}^{4}}e^{-i\theta}\right)\\
 & =\frac{1}{\tau^{4}}\Im\left(-e^{2i\theta}-2a\tau e^{i\theta}-3z\tau^{2}\right)\\
 & =\frac{1}{\tau^{4}}\left(-\sin2\theta-2a\tau\sin\theta\right)\\
 & =\frac{2\sin\theta}{\tau^{4}}\left(-\cos\theta-a\tau\right).
\end{align*}
In the formula of $t_{0}$ in Lemma \ref{lem:zerosofP(t)}, we substitute
$\tau=-1/\zeta-2\cos\theta$ and obtain

\begin{align*}
\frac{f(t_{0})-f(t_{1})e^{2i\theta}}{2it_{0}e^{2i\theta}} & =\frac{2\sin\theta}{\tau^{4}}\left(-\cos\theta+a/\zeta+2a\cos\theta\right).
\end{align*}
We finish this lemma by showing that $-\cos\theta+a/\zeta+2a\cos\theta>0$,
which implies $f(t_{0})=\overline{f(t_{1})}\ne0$ and and the lemma
will follow from \eqref{eq:diffztheta}. We expand and divide both
sides of \eqref{eq:quadraticzetatheta} by $\zeta$ to get 
\[
\zeta(1-4a\cos^{2}\theta)+2\cos\theta(1-2a)-a/\zeta=0,
\]
or equivalently, 
\[
\zeta(1-4a\cos^{2}\theta)+\cos\theta(1-2a)=-\cos\theta+2a\cos\theta+a/\zeta.
\]
Finally, using the definition of $\zeta$ in \eqref{eq:zetaform}
and Lemma \ref{lem:realzeta}, we note that 
\[
\zeta(1-4a\cos^{2}\theta)+\cos\theta(1-2a)=\sqrt{(1-4a)\cos^{2}\theta+a}>0.
\]
The proof is complete.
\end{proof}
\begin{lem}
\label{lem:zonto}The function $z(\theta)$ as defined in \eqref{eq:ztheta}
maps the interval $(2\pi/3,\pi)$ onto the interior of $I_{a}$. 
\end{lem}
\begin{proof}
Since $z(\theta)$ is a continuous increasing function on $(2\pi/3,\pi)$,
we only need to evaluate the limits of $z(\theta)$ at the endpoints.
Since $|\zeta|>1$ by Lemma \ref{lem:outsidedisk}, the formula of
$\zeta(\theta)$ in \eqref{eq:zetaform} implies that $\zeta(\theta)\rightarrow1^{+}$
as $\theta\rightarrow(2\pi/3)^{+}$. Consequently, \eqref{eq:altzform}
gives
\[
\lim_{\theta\rightarrow(2\pi/3)^{+}}z(\theta)=-\infty.
\]
Finally, from the fact that
\[
\lim_{\theta\to\pi}\zeta(\theta)=\frac{1-2a+\sqrt{1-3a}}{1-4a},
\]
and \eqref{eq:ztheta}, we have
\begin{align}
\lim_{\theta\to\pi}z(\theta) & =\frac{a(1-2a+\sqrt{1-3a})(1-4a)}{(-1+6a+\sqrt{1-3a})(-1-2\sqrt{1-3a})}\nonumber \\
 & =\frac{a(-1+4a)^{2}(-2+9a)+2a(-1+3a)(-1+4a)^{2}\sqrt{1-3a}}{27(1-4a)^{2}a}\label{eq:limitzatpi}\\
 & =\frac{-2+9a-2\sqrt{(1-3a)^{3}}}{27},\nonumber 
\end{align}
where we obtain \eqref{eq:limitzatpi} by multiply and divide the
fraction by $\left(-1+6a-\sqrt{1-3a}\right)\left(-1+2\sqrt{1-3a}\right)$. 
\end{proof}
Before making final arguments to connect all results in this section,
we check the sign of $g_{m}(\theta)$ at one of the endpoints. 
\begin{lem}
\label{lem:signend}If $-1\le a<1/4$ then the sign of $g_{m}(\pi^{-})$
is $(-1)^{m}$.
\end{lem}
\begin{proof}
Since $-1\le a<1/4$, one can check that 
\[
\lim_{\theta\to\pi^{-}}\zeta(\theta)=\frac{1-2a+\sqrt{1-3a}}{1-4a}\geq1.
\]
The result then follows directly from \eqref{eq:gthetaform} and the
fact that
\[
\lim_{\theta\to\pi^{-}}\frac{\sin(m+1)\theta}{\sin(\theta)}=(m+1)(-1)^{m}.
\]
\end{proof}
With all the lemmas at our disposal, we produce the final arguments
to finish the proof of Theorem \ref{thm:cneq0}. We consider the function
$g_{m}(\theta)$ at the points 
\[
\theta_{h}=\frac{h\pi}{m+1}\in\left(\frac{2\pi}{3},\pi\right),\qquad\left\lfloor \frac{2(m+1)}{3}\right\rfloor +1\le h\le m.
\]
We note that the number of such values of $h$ is 
\[
m-\left\lfloor \frac{2(m+1)}{3}\right\rfloor =\left\lfloor \frac{m}{3}\right\rfloor ,
\]
where the equality can be checked by considering $m$ congruent to
$0,$1, or $2$ modulus $3$. From the formula of $g_{m}(\theta)$
in \eqref{eq:gthetaform} and Lemma \ref{lem:outsidedisk}, the sign
of $g_{m}(\theta_{h})$ is $(-1)^{h-1}$. By the Intermediate Value
Theorem and Lemma \ref{lem:singinterval}, there are at least $\left\lfloor m/3\right\rfloor -1$
zeros of $g_{m}(\theta)$ on $(2\pi/3,\pi)$. In fact, we claim that
there are at least $\left\lfloor m/3\right\rfloor $ zeros of $g_{m}(\theta)$
on $(2\pi/3,\pi)$. In the case $-1\le a<1/4$, we obtain one more
zero of $g_{m}(\theta)$ from Lemma \ref{lem:signend}. On the other
hand, if $1/4<a\le1/3$, then we obtain another zero of $g_{m}(\theta)$
by Lemma \ref{lem:singinterval}. From Lemmas \ref{lem:zmonotone}
and \ref{lem:zonto}, we obtain at least $\left\lfloor m/3\right\rfloor $
zeros of $H_{m}(z)$ on $I_{a}$. Since the degree of $H_{m}(z)$
is at most $\left\lfloor m/3\right\rfloor $ by Lemma \ref{degreelem},
all the zeros of $H_{m}(z)$ lie on $I_{a}$. Recall that we can ignore
the case $a=1/4$ by Lemma \ref{lem:densesubint}. The density of
${\displaystyle \bigcup_{m=0}^{\infty}\mathcal{Z}(H_{m}(z))}$ on
$I_{a}$ comes from the density of ${\displaystyle \bigcup_{m=0}^{\infty}\mathcal{Z}(g_{m}(\theta))}$
on $(2\pi/3,\pi)$ and from $z(\theta)$ being a continuous map.

\section{The case $c=0$ and $b\ge0$}

It is trivial that if $c=0$ and $b=0$ then the zeros of $H_{m}(z)$
are real under the convention that the constant zero polynomial is
hyperbolic. In the case $b>0$, we make a substitution $t\rightarrow t/\sqrt{b}$
and reduce the problem to the following theorem. 
\begin{thm}
\label{thm:ceq0}The zeros of the sequence of polynomials $H_{m}(z)$
generated by 
\begin{equation}
\sum_{m=0}^{\infty}H_{m}(z)t^{m}=\frac{1}{1+t^{2}+zt^{3}}\label{eq:genceq0}
\end{equation}
are real, and the set ${\displaystyle \bigcup_{m=0}^{\infty}\mathcal{Z}(H_{m})}$
is dense on $(-\infty,\infty)$. 
\end{thm}
The proof of Theorem \ref{thm:ceq0} follows from a similar procedure
seen in Section 2. We will point out some key differences. The following
lemma comes directly from the recurrence relation 
\[
H_{m}(z)+H_{m-2}(z)+zH_{m-3}(z)=0
\]
and induction. 
\begin{lem}
The degree of the polynomial $H_{m}(z)$ generated by \eqref{eq:genceq0}
is at most 
\[
\begin{cases}
\frac{m}{3} & \text{ if }m\equiv0\pmod3\\
\frac{m-4}{3} & \text{ if }m\equiv1\pmod3\\
\frac{m-2}{3} & \text{ if }m\equiv2\pmod3.
\end{cases}
\]
\end{lem}
We define the functions 
\begin{equation}
\begin{aligned}\zeta(\theta) & =-\frac{1}{2\cos\theta},\\
g_{m}(\theta) & =\frac{-\sin(m+1)\theta}{2\cos\theta\sin\theta}\left(2+\cos2\theta\right)-\cos(m+1)\theta+(-2\cos\theta)^{m+1},\qquad\text{and}\\
z(\theta) & =\frac{2\cos\theta}{\sqrt{(1-4\cos^{2}\theta)^{3}}}
\end{aligned}
\label{eq:defsceq0}
\end{equation}
on the interval $(\pi/3,\pi/2)$.

The proof of the lemma below is similar to that of Lemma \ref{lem:zerosofP(t)}.
We leave the detailed computations to the reader. 
\begin{lem}
Suppose $\theta\in(\pi/3,\pi/2)$, $\zeta=\zeta(\theta)$, and $z=z(\theta)$
defined by \eqref{eq:defsceq0}. The three zeros of $1+t^{2}+z(\theta)t^{3}$
are
\[
t_{0}=-\frac{e^{-i\theta}}{z(2\cos\theta+\zeta)},
\]
\[
t_{1}=t_{0}e^{2i\theta},
\]

\[
t_{2}/t_{0}=\zeta e^{i\theta}.
\]
\end{lem}
Looking at $z'(\theta)$, one can check that $z(\theta)$ is strictly
decreasing on the interval $(\pi/3,\pi/2)$. By the partial fraction
decomposition of 
\[
\sum_{m=0}^{\infty}H_{m}(z)t^{m}=\frac{1}{1+t^{2}+zt^{3}}=\frac{1}{z(t-t_{0})(t-t_{1})(t-t_{2})}
\]
similar to the previous section, we conclude that for each zero of
$g_{m}(\theta)$ on the interval $(\pi/3,\pi/2)$ we obtain two zeros
$\pm z(\theta)$ of $H_{m}(z)$. We can also check by induction that
$z=0$ is a simple zero of $H_{m}(z)$ if $m$ is odd, and $z=0$
is not a zero of $H_{m}(z)$ when $m$ is even. The formula of $g_{m}(\theta)$
implies that the sign of this function alternates when $\cos(m+1)\theta=\pm1$,
that is, 
\[
(m+1)\theta=k\pi,\qquad\frac{m+1}{3}<k<\frac{m+1}{2}.
\]
 Since $g_{m}(\theta)$ is continuous on $(\pi/3,\pi/2)$, we may
apply the Intermediate Value Theorem to compute the number of zeros
of $g_{m}(\theta)$ on $(\pi/3,\pi/2)$ and correspond to the number
of zeros of $H_{m}(z)$ on $(-\infty,\infty)$. We will see that this
number is equal to the degree of $H_{m}(z)$ and Theorem \ref{thm:ceq0}
follows. We quickly summarize in the six cases below where $\theta^{*}$
denotes the smallest solution $(m+1)\theta=k\pi$ on the interval
$(\pi/3,\pi/2)$. 
\begin{enumerate}
\item $m\equiv1\pmod3$ and $m$ is even: there are 
\[
\frac{m}{2}-\frac{m+2}{3}=\frac{m-4}{6}
\]
 zeros of $g_{m}(\theta)$ on $(\pi/3,\pi/2)$, which gives $(m-4)/3$
zeros of $H_{m}(z)$ on $(-\infty,\infty)$.
\item $m\equiv1\pmod3$ and $m$ is odd: there are 
\[
\frac{m-1}{2}-\frac{m+2}{3}=\frac{m-7}{6}
\]
zeros of $g_{m}(\theta)$ on $(\pi/3,\pi/2)$ which gives$(m-7)/3$
nonzero roots of $H_{m}(z)$. We add a simple zero $z=0$ and obtain
$(m-4)/3$ zeros of $H_{m}(z)$ on $(-\infty,\infty)$. 
\item $m\equiv0\pmod3$ and $m$ is even: with the observation that $\lim_{\theta\rightarrow\pi/3}g_{m}(\theta)=-3<0$
and $g_{m}(\theta^{*})>0$ we obtain 
\[
\frac{m}{2}-\left(\frac{m}{3}+1\right)+1=\frac{m}{6}
\]
zeros of $g_{m}(\theta)$ on $(\pi/3,\pi/2)$, which gives $m/3$
zeros of $H_{m}(z)$ on $(-\infty,\infty)$. 
\item $m\equiv0\pmod3$ and $m$ is odd: with the observation that $\lim_{\theta\rightarrow\pi/3}g_{m}(\theta)=3>0$
and $g_{m}(\theta^{*})<0$ we obtain
\[
\frac{m-1}{2}-\left(\frac{m}{3}+1\right)+1=\frac{m-3}{6}
\]
zeros of $g_{m}(\theta)$ on $(\pi/3,\pi/2)$, which gives $(m-3)/3$
nonzero roots of $H_{m}(z)$. We add a simple zero $z=0$ and obtain
$m/3$ zeros of $H_{m}(z)$ on $(-\infty,\infty)$.
\item $m\equiv2\pmod3$ and $m$ is even: with the observation that $g_{m}(\pi/3)=0$,
$g_{m}'(\pi/3)>0$, and $g_{m}(\theta^{*})<0$ we obtain 
\[
\frac{m}{2}-\left(\frac{m+1}{3}+1\right)+1=\frac{m-2}{6}
\]
zeros of $g_{m}(\theta)$ on $(\pi/3,\pi/2)$, which gives $(m-2)/3$
zeros of $H_{m}(z)$ on $(-\infty,\infty)$.
\item $m\equiv2\pmod3$ and $m$ is odd: with the observation that $g_{m}(\pi/3)=0$,
$g_{m}'(\pi/3)<0$, and $g_{m}(\theta^{*})>0$ we obtain 
\[
\frac{m-1}{2}-\left(\frac{m+1}{3}+1\right)+1=\frac{m-5}{6}
\]
zeros of $g_{m}(\theta)$ on $(\pi/3,\pi/2)$, which gives $(m-5)/3$
nonzero roots of $H_{m}(z)$. We plus add simple zero $z=0$ and obtain
$(m-2)/3$ zeros of $H_{m}(z)$ on $(-\infty,\infty)$. 
\end{enumerate}
In all cases above the number of zeros of $H_{m}(z)$ on $(-\infty,\infty)$
corresponds to the degree of $H_{m}(z)$ and Theorem \ref{thm:ceq0}
follows.

\section{Necessary condition for the reality of zeros}

To prove the necessary condition of Theorem \ref{thm:mainstatement},
we first show that if $c=0$ and $b<0$ then not all polynomials $H_{m}(z)$
are hyperbolic. In fact, with the substitution $t$ by $it$, we conclude
that all the zeros of $H_{m}(z)$ will be purely imaginary by Theorem
\ref{thm:ceq0}.

It remains to consider the sequence $H_{m}(z)$ generated by 
\[
\sum_{m=0}^{\infty}H_{m}(z)t^{m}=\frac{1}{1+t+at^{2}+zt^{3}}
\]
and show that if $a\notin[-1,1/3]$ then there is an $m$ such that
not all the zeros of $H_{m}(z)$ are real. In fact, we will show if
$a\notin[-1,1/3]$, $H_{m}(z)$ is not hyperbolic for all large $m$.
To prove this, let us introduce some definitions (discussed in \cite{sokal})
related to the root distribution of a sequence of functions 
\[
f_{m}(z)=\sum_{k=1}^{n}\alpha_{k}(z)\beta_{k}(z)^{m}
\]
where $\alpha_{k}(z)$ and $\beta_{k}(z)$ are analytic in a domain
$D$. We say an index $k$ dominant at $z$ if $|\beta_{k}(z)|\ge|\beta_{l}(z)|$
for all $l$ ($1\le l\le n$). Let 
\[
D_{k}=\{z\in D:k\mbox{ is dominant at }z\}.
\]
Let $\liminf\mathcal{Z}(f_{m})$ be the set of all $z\in D$ such
that every neighborhood $U$ of $z$ has a non-empty intersection
with all but finitely many of the sets $\mathcal{Z}(f_{m})$. Let
$\limsup\mathcal{Z}(f_{m})$ be the set of all $z\in D$ such that
every neighborhood $U$ of $z$ has a non-empty intersection infinitely
many of the sets $\mathcal{Z}(f_{m})$. We will need the following
theorem from Sokal (\cite[Theorem 1.5]{sokal}). 
\begin{thm}
\label{sokal}Let $D$ be a domain in $\mathbb{C}$, and let $\alpha_{1},\ldots,\alpha_{n},\beta_{1},\ldots,\beta_{n}$
$(n\ge2)$ be analytic function on $D$, none of which is identically
zero. Let us further assume a 'no-degenerate-dominance' condition:
there do not exist indices $k\ne k'$ such that $\beta_{k}\equiv\omega\beta_{k'}$
for some constant $\omega$ with $|\omega|=1$ and such that $D_{k}$
$(=D_{k'})$ has nonempty interior. For each integer $m\ge0$, define
$f_{m}$ by 
\[
f_{m}(z)=\sum_{k=1}^{n}\alpha_{k}(z)\beta_{k}(z)^{m}.
\]
Then $\liminf Z(f_{m})=\limsup Z(f_{m})$, and a point $z$ lies in
this set if and only if either

(i) there is a unique dominant index $k$ at $z$, and $\alpha_{k}(z)=0$,
or

(ii) there a two or more dominant indices at $z$. 
\end{thm}
If $z^{*}\in\mathbb{C}$ such that the zeros in $t$ of $1+t+at^{2}+z^{*}t^{3}$
are distinct then by partial fractions given in \eqref{eq:partialfractions}
and Theorem \ref{sokal}, $z^{*}$ will belong to $\liminf\mathcal{Z}(H_{m})$
when the two smallest (in modulus) roots of $1+t+at^{2}+z^{*}t^{3}$
have the same modulus. We also note that $t_{0}(z)$, $t_{1}(z)$,
and $t_{2}(z)$ are analytic in a neighborhood of $z^{*}$ by the
Implicit Function Theorem. If we let $\omega=e^{2i\theta}$ then the
'no-degenerate-dominance' condition in Theorem \ref{sokal} comes
directly from equations \eqref{eq:ztheta} and \eqref{eq:zetaform}
since $\theta$ is a fixed constant (and thus $z$ is a fixed point
which has empty interior). 

Suppose $a\notin[-1,1/3]$. With the arguments in the previous paragraph,
our main goal is to find a $z^{*}\notin\mathbb{R}$ so that the zeros
of $1+t+at^{2}+z^{*}t^{3}$ are distinct and the two smallest (in
modulus) zeros of this polynomial have the same modulus. If we can
find such a point, then $z^{*}\in\liminf\mathcal{Z}(H_{m})=\limsup\mathcal{Z}(H_{m})$.
This implies that on a small neighborhood of $z^{*}$ which does not
intersect the real line, there is a non-real zero of $H_{m}(z)$ for
all large $m$ by the definition of $\liminf\mathcal{Z}(H_{m})$.
Our choice of $z^{*}=z(\theta^{*})$ comes from \eqref{eq:ztheta}
for a special $\theta^{*}$. Unlike in Section 2, $\theta^{*}$ will
not belong to $(2\pi/3,\pi)$ to ensure that $z^{*}\notin\mathbb{R}.$
In particular, we consider the two cases $a<-1$ and $a>3$.

\subsection*{Case $a<-1$. }

We will pick $\theta^{*}\ne\pi/2$ but sufficiently close to $\pi/2$
on the left. Since from \eqref{eq:zetaform}
\[
\lim_{\theta\to\pi/2}\zeta(\theta)=i\sqrt{\left|a\right|},
\]
we can pick $0<\theta^{*}<\pi/2$ sufficiently close to $\pi/2$ so
that $\zeta:=\zeta(\theta^{*})\in\mathbb{C}\backslash\mathbb{R}$
and $\left|\zeta(\theta^{*})\right|>1$. By Lemma \ref{lem:zerosofP(t)},
we have $t_{2}=\zeta t_{0}e^{i\theta^{*}}$ and $t_{1}=t_{0}e^{2i\theta^{*}}$.
The fact that $|\zeta|>1$ and $\theta^{*}\ne0,\pi/2$ implies that
the polynomial $1+t+at^{2}+z(\theta^{*})t^{3}$ have distinct zeros
and not all its zeros are real. We will show that $z(\theta^{*})\notin\mathbb{R}$
by contradiction. In deed, if $z(\theta^{*})\in\mathbb{R}$ then the
zeros of the polynomial $1+t+at^{2}+z(\theta^{*})t^{3}\in\mathbb{R}[t]$
satisfy $t_{0}=\overline{t_{1}}$ and
\[
t_{2}=t_{0}\zeta e^{i\theta^{*}}\in\mathbb{R}.
\]
This gives a contradiction because the first equation implies that
$t_{0}e^{i\theta^{*}}\in\mathbb{R}$, while the second equation implies
that $t_{0}e^{i\theta^{*}}\notin\mathbb{R}$ since $\zeta\notin\mathbb{R}$. 

\subsection*{Case $a>1/3$.}

We will pick $\theta^{*}$ so that $\cos\theta^{*}=\beta^{+}$ where
$\beta=\sqrt{a/(4a-1)}<1$. The definition of $\zeta(\theta)$ in
\eqref{eq:zetaform} gives
\[
\left|\lim_{\cos\theta\to\beta}\zeta(\theta)\right|=\begin{cases}
\left|\frac{\sqrt{4a^{2}-a}}{1-2a}\right| & \text{if }a\ne1/2\\
\infty & \text{if }a=1/2
\end{cases}
\]
where we can easily check that
\[
\left|\frac{\sqrt{4a^{2}-a}}{1-2a}\right|>1,\qquad a>1/3.
\]
Thus if $\cos\theta^{*}=\beta^{+}$ then $0<\theta^{*}<\pi/2$, $|\zeta(\theta^{*}|>1$,
and $|\zeta(\theta^{*})|\notin\mathbb{R}$ where the last statement
comes from \eqref{eq:zetaform} and the inequality 
\[
(1-4a)\cos^{2}\theta^{*}+a<0.
\]
With $0<\theta^{*}<\pi/2$, $|\zeta(\theta^{*}|>1$, and $|\zeta(\theta^{*})|\notin\mathbb{R}$,
we apply the same arguments given in the previous case to complete
this section. 

\subsection*{Acknowledgments }

The authors would like to thank Professor T. Forg\'acs for his careful
review of the paper.


\begin{thebibliography}{10}
\bibitem[1]{bkw}S. Beraha, J. Kahane, N. J. Weiss, Limits of zeroes
of recursively defined polynomials, Proc. Nat. Acad. Sci. U.S.A. 72
(1975), no. 11, 4209.

\bibitem[2]{bkw-1}S. Beraha, J. Kahane, N. J. Weiss, Limits of zeros
of recursively defined families of polynomials, Studies in foundations
and combinatorics, pp. 213\textendash 232, Adv. in Math. Suppl. Stud.,
1, Academic Press, New York-London, 1978.

\bibitem[3]{bbs}Borcea J., Bøgvad R., Shapiro B., On rational approximation
of algebraic functions, Adv. Math. 204 (2006), no. 2, 448\textendash 480.

\bibitem[4]{bb}{]} J. Borcea and P. Brändén, Pólya-Schur master theorems
for circular domains and their boundaries, Annals of Math., 170 (2009),
465-492.

\bibitem[5]{bg}R. Boyer, W. M. Y. Goh, On the zero attractor of the
Euler polynomials, Adv. in Appl. Math. 38 (2007), no. 1, 97\textendash 132.

\bibitem[6]{bg-1}R. Boyer, W. M. Y. Goh, Polynomials associated with
partitions: asymptotics and zeros, Special functions and orthogonal
polynomials, 33\textendash 45, Contemp. Math., 471, Amer. Math. Soc.,
Providence, RI, 2008.

\bibitem[7]{fp}T. Forgács and A. Piotrowski, Hermite multiplier sequences
and their associated operators, Constr. Approx. 43(3) (2015), pp.
459-479.

\bibitem[8]{ft-1}T. Forgács , K. Tran , Polynomials with rational
generating functions and real zeros, J. Math. Anal. Appl. 443 (2016)
pp. 631-651.

\bibitem[9]{sokal}A. Sokal, Chromatic roots are dense in the whole
complex plane, Combin. Probab. Comput. 13 (2004), no. 2, 221\textendash 261.

\bibitem[10]{tran}K. Tran , Connections between discriminants and
the root distribution of polynomials with rational generating function,
J. Math. Anal. Appl. 410 (2014), 330-340.

\bibitem[11]{tran-1}K. Tran, The root distribution of polynomials
with a three-term recurrence, J. Math. Anal. Appl. 421 (2015), 878\textendash 892. 
\end{thebibliography}
\end{document}